\documentclass{amsart}

\usepackage[english]{babel}
\usepackage[latin1]{inputenc}
\usepackage{amsmath}
\usepackage{amsthm}
\usepackage{amssymb}
\usepackage{tikz}

\usepackage{imakeidx}
\usepackage{appendix}
\usepackage{tikz-cd}
\usepackage{verbatim}
\usepackage{enumitem}
\usepackage{multicol}

\usepackage{hyperref}
\usepackage{cleveref}

\usepackage{mathtools}

\newtheorem{thm}{Theorem}[section]

\newtheorem{prop}[thm]{Proposition}

\newtheorem{lemma}[thm]{Lemma}
\newtheorem{cor}[thm]{Corollary}

\theoremstyle{definition}
\newtheorem{defin}[thm]{Definition}

\theoremstyle{remark}

\numberwithin{equation}{section}

\newcommand{\Q}{\mathbb Q}
\newcommand{\Z}{\mathbb Z}
\newcommand{\G}{\mathbb G}

\renewcommand{\c}{\subseteq}
\newcommand{\A}{\mathbb A}

\newcommand{\cl}{\overline}

\renewcommand{\phi}{\varphi}
\newcommand{\on}[1]{\operatorname{#1}}

\DeclareFontFamily{U}{wncy}{}
\DeclareFontShape{U}{wncy}{m}{n}{<->wncyr10}{}
\DeclareSymbolFont{mcy}{U}{wncy}{m}{n}
\DeclareMathSymbol{\Sh}{\mathord}{mcy}{"58}

\title{Retract rationality and algebraic tori}
\author{Federico Scavia}

\begin{document}
	
	\begin{abstract} For any prime number $p$ and field $k$, we characterize the $p$-retract rationality of an algebraic $k$-torus in terms of its character lattice. We show that a $k$-torus is retract rational if and only if it is $p$-retract rational for every prime $p$, and that the Noether problem for retract rationality for a group of multiplicative type $G$ has an affirmative answer for $G$ if and only if the Noether problem for $p$-retract rationality for $G$ has a positive answer for all $p$. For every finite set of primes $S$ we give examples of tori that are $p$-retract rational if and only if $p\notin S$. \end{abstract}
	\maketitle	
	\section{Introduction}
	
	A classical question in algebraic geometry is to decide whether a given variety is rational, that is, birational to some affine space. Questions of this type are called "rationality problems". For a given variety it is often easier to establish weaker properties, such as unirationality, stable rationality or retract rationality. However, deciding under what circumstances these weaker properties imply rationality leads to difficult questions. In particular, the L\"uroth problem asks whether every unirational variety is rational, and the Zariski problem asks whether every stably rational variety is rational. The Noether problem asks whether, for a given algebraic group $G$ and generically free finite-dimensional linear $G$-representation $V$, the rational quotient variety $V/G$ is rational. 	
	
	The first constructions of varieties that are unirational but not rational (over $\mathbb{C}$) were given by Clemens-Griffiths \cite{clemens1972intermediate} and Manin-Iskovskikh \cite{iskovskih1971three}, thus answering the L\"uroth problem in the negative. Unramified cohomology, defined by Artin-Mumford \cite{artin1972some} and
	Saltman~\cite{saltman1984retract}, was used to construct further examples of unirational varieties that are not rational. In particular, Saltman showed that the Noether problem over $\mathbb C$ has a negative answer for some finite groups $G$.
	The first counterexamples to the Zariski problem were constructed by Beauville, Colliot-Th\'{e}l\`ene, Sansuc and Swinnerton-Dyer in~\cite{beauville1985varietes}. There has been much recent work on the rationality problem for projective hypersurfaces, using the degeneration method, due to Voisin \cite{voisin2015unirational} and subsequently generalized by Colliot-Th\'{e}l\`ene-Pirutka \cite{colliot2016hypersurfaces} and others. For further developments, see~\cite{totaro2016hypersurfaces} \cite{schreieder2018stably}.
	
	Recall that a variety $X$ is defined to be retract rational if the identity map of $X$ factors rationally through some affine space. If $p$ is a prime, $X$ is defined to be $p$-retract rational if there exists a diagram
	\[
	\begin{tikzcd}
	& X' \arrow[d, "f"] \arrow[dl] \\
	\A^n \arrow[r, dashrightarrow] & X  
	\end{tikzcd}
	\]\noindent
	where $f$ is a finite dominant morphism of degree not divisible by $p$. The definition of retract rationality is due to Saltman \cite{saltman1982generic}. The notion of $p$-retract rationality, recently introduced by Merkurjev \cite{merkurjev2018versal}, naturally leads to the following question: if $X$ is a $p$-retract rational variety for every prime $p$, is it retract rational? In other words, is retract rationality a $p$-local property? As Merkurjev shows in \cite[Corollary 7.5]{merkurjev2018versal}, unramified cohomology cannot tell the difference between the class of retract rational varieties and that of varieties which are $p$-retract rational for every prime $p$. 
	
	In this note we consider the notion of $p$-retract rationality in the context of algebraic tori. The birational geometry of tori has been intensely studied by Voskresenski\u{\i} \cite{voskresenskii1965two} \cite{voskresenskii1971rationality} \cite{voskresenskii1973fields}, Endo-Miyata \cite{miyata1971invariants} \cite{endo1973invariants} \cite{endo1975classification}, Colliot-Th\'el\`ene-Sansuc \cite{colliot1977r} \cite{colliot1987principal}, Kunyavski\u{\i} \cite{kunyavskii1978tori} \cite{kunyavskii1987three} and many others; the standard reference on this material is \cite{voskresenskii2011algebraic}. 
	
	If $k$ is a field, an algebraic $k$-torus is determined by its character lattice, viewed as an integral representation of the absolute Galois group of $k$. Birational properties of algebraic tori frequently translate to properties of their character lattice, with the notable exception of rationality. The Zariski problem in the case of tori is known as Voskresenski\u{\i}'s Conjecture; see \cite[p. 68]{voskresenskii2011algebraic}.	
	
	The main technical result of this paper is a criterion for $p$-retract rationality of a $k$-torus $T$ in terms of its character lattice $\hat{T}$; see~\Cref{character}.
	As a consequence of this description we will show that retract rationality of tori is a $p$-local property in the following sense.
	
	\begin{thm}\label{local}
		Let $T$ be an algebraic torus. Then $T$ is retract rational if and only if it is $p$-retract rational for every prime $p$.
	\end{thm}
	
	
	Let $G$ be a linear algebraic group, and let $V$ be a finite-dimensional generically free representation of $G$. There exists a dense $G$-invariant open subset $U$ of $V$ such that the geometric quotient $U/G$ exists and $U\to U/G$ is a $G$-torsor. One may regard $U/G$ as a variety approximating the classifying stack $BG$. By the no-name lemma \cite[Lemma 2.1]{reichstein2006birational}, the stable (retract, $p$-retract) rationality of $U/G$ does not depend on the representation $V$ but only on $G$. We say that $BG$ is stably (retract, $p$-retract) rational if so is $U/G$. In different but equivalent terminology, this means that the Noether problem for stable (retract, $p$-retract) rationality has an affirmative answer; see \cite[\S 3]{florence2018genus0}. \Cref{local} has the following consequence.
	
	\begin{cor}\label{bglocal}
		Let $G$ be a group of multiplicative type over $k$. If $BG$ is $p$-retract rational for every prime $p$, then it is retract rational.
	\end{cor}
	
	One might ask if the retract rationality of $BG$ is a $p$-local property, in the case when $G$ is an arbitrary linear algebraic group. A positive answer would have profound implications for the Noether problem.
	
	For example, if $A_n$ denotes the alternating group on $n$ elements, Merkurjev shows that the classifying stack $BA_n$ is $p$-retract rational for every prime $p$; see \cite[Theorem 6.1]{merkurjev2018versal}. On the other hand, retract rationality of $BA_n$ is a classical problem due to Hilbert, which remains open for every $n\geq 6$; see \cite{maeda1989noether} and \cite[\S 4.7]{colliot2007rationality}.
	
	We also give examples of tori that are not $p$-retract rational at a specified finite set of primes.
	\begin{thm}\label{Theorem 1.3}
		Let $S$ be a finite set of primes. There exists an algebraic torus $T$ such that $T$ is $p$-retract rational if and only if $p\notin S$.
	\end{thm}
	We note that any unirational $k$-variety $X$ (in particular, any algebraic torus) is $p$-retract rational for all but at most finitely primes $p$. If $X$ is unirational, one can find a generically finite dominant rational map $\phi:\A^n\dashrightarrow X$ of some degree $d\geq 1$; see \cite[Lemma 1]{roquette1964isomorphisms} for infinite $k$, and \cite{ohm1984subfields} for arbitrary $k$. If $U\c \A^n$ is a non-empty open subset such that $\phi$ is defined over $U$ and is finite of degree $d$, the diagram
	\[
	\begin{tikzcd}
	& U \arrow[d]  \\
	\A^n \arrow[ur, hookleftarrow]\arrow[r, dashrightarrow,"\phi"] & X  
	\end{tikzcd}
	\]\noindent
	shows that $X$ is $p$-retract rational for every prime $p$ not dividing $d$.
	In the case of tori, the assertion also follows from \Cref{firstprop}\ref{firstprop1b} and \Cref{character}.
	
	The rest of the paper is structured as follows. In \Cref{sec2} we define $p$-invertible lattices and establish their basic properties. We use this notion to characterize $p$-retract rationality of tori in \Cref{sec3}, and to prove Theorem \ref{local} and Corollary \ref{bglocal} in \Cref{sec4}. In \Cref{sec5} we prove \Cref{Theorem 1.3}, and deduce a consequence for the Noether problem in \Cref{stackspecified}. In the Appendix, we insert a proof of an auxiliary result due to Colliot-Th\'el\`ene but so far unpublished.
	
	\section{$p$-invertible lattices}\label{sec2}
	
	Throughout this paper we will denote by $p$ a prime number. If $T$ is a torus over a field $k$, we denote by $\hat{T}$ its character lattice. If $M$ is a lattice, we set $M_{(p)}:=M\otimes_{\Z}\Z_{(p)}$, where $\Z_{(p)}$ is the localization of $\Z$ at the prime ideal $(p)$. If $d$ is an integer, we denote by $\phi_{M,d}:M\to M$ 
	the endomorphism of multiplication by $d$. 
	
	Recall that, if $G$ is a profinite group, a $G$-lattice $M$ is defined to be invertible (or permutation projective) if it is a direct summand of a permutation $G$-lattice.
	\begin{defin}\label{pproj}
		Let $G$ be a profinite group and let $M$ be a $G$-lattice. We say that $M$ is $p$-invertible (or $p$-permutation projective) if there exists a permutation $G$-lattice $P$ such that $M_{(p)}$ is a direct summand of $P_{(p)}$. In other words, $M$ is $p$-invertible if there exists a commutative diagram of $G$-lattices
		\begin{equation}\label{pprojdiag1}
		\begin{tikzcd}
		M \arrow[r,"\iota"] \arrow[dr,swap,"\phi_{M,d}"]& P \arrow[d,"\pi"]\\
		& M 
		\end{tikzcd}
		\end{equation}\noindent 
		for some $d$ not divisible by $p$.
	\end{defin}
	It follows from the definition that $\iota$ is injective. It is clear that if $M$ is invertible, it is $p$-invertible for every prime $p$.
	
	If $G$ is a profinite group and $M$ is a $G$-lattice, the $G$-action on $M$ factors through a finite quotient $G'$. The next lemma assures us that $M$ is $p$-invertible as a $G$-lattice if and only if it is $p$-invertible as a $G'$-lattice.
	
	\begin{lemma}
		Let $G$ be a profinite group and $M$ be a $G$-lattice. Let $H$ be a closed subgroup of $G$ acting trivially on $M$. Then $M$ is $p$-invertible as a $G$-lattice if and only if it is $p$-invertible as a $G/H$-lattice.
	\end{lemma}
	
	\begin{proof}
		Assume that $\phi_{M,d}$ factors through a permutation $G$-lattice $P$ for some $d$ not divisible by $p$. By \cite[Lemme 2(i)]{colliot1977r}, $P^H$ is a permutation $G/H$-lattice. Since $\phi_{M,d}$ factors through $P^H$, $M$ is $p$-invertible as a $G/H$-lattice.
		
		Conversely, assume that $\phi_{M,d}$ factors through a permutation $G/H$-lattice $P$ for some $d$ not divisible by $p$. By \cite[Lemme 2(iv)]{colliot1977r}, $P$ is also a permutation $G$-lattice, so $M$ is $p$-invertible as a $G$-lattice as well.
	\end{proof}
	
	\begin{lemma}\label{firstprop}
		Let $G$ be a finite group and let $M$ be a $G$-lattice.
		\begin{enumerate}[label=(\alph*)]
			\item\label{firstprop1b} If $p$ is prime and $p\nmid |G|$, then $M$ is $p$-invertible.
			\item\label{firstprop2} If $M$ is $p$-invertible for every prime $p$, then $M$ is invertible.
			\item\label{firstprop3} Assume that there exists a diagram 
			\begin{equation}\label{pprojdiag2}
			\begin{tikzcd}
			M \arrow[r,"\iota"] \arrow[dr,swap,"\phi"]& P \arrow[d,"\pi"]\\
			& M' 
			\end{tikzcd}
			\end{equation}\noindent 
			where $P$ is a permutation $G$-lattice, $\phi$ is injective and $\on{coker}\phi$ is finite of order not divisible by $p$. Then $M$ is $p$-invertible.
			\item\label{firstprop4} Let $M$ and $N$ be two stably equivalent $G$-lattices, and assume that $M$ is $p$-invertible. Then $N$ is $p$-invertible.
		\end{enumerate}
	\end{lemma}
	Recall that two $G$-lattices $M$ and $N$ are said to be stably equivalent if there exist permutation $G$-lattices $P$ and $Q$ such that $M\oplus P\cong N\oplus Q$.
	\begin{proof}
		\ref{firstprop1b}. Fix an embedding $\iota:M\hookrightarrow P$, where $P$ is a permutation $G$-lattice. The standard proof of Maschke's Theorem shows the existence of a homomorphism of $\Q [G]$-modules $\pi:P_{\Q}\to M_{\Q}$ such that $\pi\circ\iota=\on{id}_{M}$ with only $|G|$ in the denominator. In other words, $\pi':=\phi_{M,|G|}\circ \pi$ is a well defined homomorphism $P\to M$. Hence $\phi_{M,|G|}=\pi'\circ\iota$ factors through $P$. If $p\nmid |G|$, this shows that $M$ is $p$-invertible.
		
		\ref{firstprop2}. By assumption, for every prime $p$ there exist an integer $d_p$ not divisible by $p$, a permutation $G$-lattice $P_p$ and homomorphisms $f_p:M\to P_p$, $g_p:P_p\to M$ such that $\phi_{M,d_p}=g_p\circ f_p$. The $d_p$ are coprime, hence we may write $1=\sum_{i=1}^ra_id_{p_i}$ for some integers $a_i$. Letting $P:=\oplus_{i=1}^rP_{p_i}$, we have that $\on{id}_M=g\circ f$, where $f=(f_{p_i}):M\to P$ and $g=\sum a_ig_{p_i}:P\to M$.
		
		
		\ref{firstprop3}. The map $\phi_{(p)}:M_{(p)}\to M'_{(p)}$ is an isomorphism, and $\on{id}_{M_{(p)}}=\phi^{-1}_{(p)}\circ\phi_{(p)}$ factors through $P_{(p)}$. Clearing denominators, this gives a $d\in \Z$ not divisible by $p$ such that $\phi_{M,d}$ factors through $P$. 
		
		\ref{firstprop4}. By assumption, $M\oplus P\cong N\oplus Q$ for some permutation $G$-lattices $P$ and $Q$. There exists a permutation $G$-lattice $R$ such that $M_{(p)}$ is a direct summand of $R_{(p)}$, hence $N_{(p)}\oplus Q_{(p)}\cong M_{(p)}\oplus P_{(p)}$ is a direct summand of $(R\oplus P)_{(p)}$, so $N_{(p)}$ is also a summand of $(R\oplus P)_{(p)}$.
	\end{proof}

	If $G$ is a finite group and $M$ is a $G$-module, for every integer $i$ we denote by $\hat{H}^i(G,M)$ the Tate cohomology group of degree $i$; see \cite[\S 2.5]{lorenz2006multiplicative} or \cite[Chapter VI]{brown1994cohomology}.
	
	\begin{lemma}\label{ext}
		Let $M$ be a $p$-invertible $G$-lattice. Then:
		\begin{enumerate}[label=(\alph*)]
			\item\label{ext1} $\hat{H}^i(G',M)_{(p)}=0$ for every subgroup $G'$ of $G$, $i=\pm 1$;
			\item\label{ext2} if $F$ is a flasque $G$-lattice, then $\on{Ext}_G^1(F,M)_{(p)}=0$.
		\end{enumerate} 
	\end{lemma}
	
	\begin{proof}
		\ref{ext1}. Assume that $\phi_{M,d}$ factors through the permutation $G$-lattice $P$, for some $d$ not divisible by $p$. We have a short exact sequence \[0\to M\xrightarrow{\phi_{M,d}} M\to A\to 0\] where $A$ is $d$-torsion. Consider the associated long exact sequence in cohomology
		\[\dots\to \hat{H}^{i-1}(G',A)\to \hat{H}^i(G',M)\xrightarrow{\hat{H}^i(\phi_{M,d})} \hat{H}^i(G',M)\to \hat{H}^i(G',A)\to\cdots.\] 
		Since $A$ is $d$-torsion, $\hat{H}^i(G',A)$ is also $d$-torsion for every $i$, so $\hat{H}^i(G',A)_{(p)}=0$. It follows that $\hat{H}^i(\phi_{M,d})_{(p)}$ is an isomorphism. On the other hand, $\hat{H}^i(\phi_{M,d})$ factors through $\hat{H}^i(G',P)$, which vanishes for $i=\pm 1$ because a permutation $G$-lattice is both flasque and coflasque, so $\hat{H}^i(\phi_{M,d})$ is the zero map for $i=\pm 1$. 
		
		\ref{ext2}. If $P$ is a permutation $G$-lattice, \cite[Lemme 9]{colliot1977r} shows that $\on{Ext}_G^1(F,P)=0$ for every flasque $G$-lattice $F$. The assertion $\on{Ext}_G^1(F,M)_{(p)}=0$ can now be proved using  the same argument as in part \ref{ext1}, with $H^i$ replaced by $\on{Ext}_G^1$.
	\end{proof}
	
	\begin{lemma}\label{sylow}
		Let $G$ be a finite group, $G_p$ a $p$-Sylow subgroup of $G$, and $M$ a $G$-lattice. The following are equivalent:
		\begin{enumerate}
			\item\label{sylow1} $M$ is $p$-invertible as a $G$-lattice;
			\item\label{sylow2} $M$ is $p$-invertible as a $G_p$-lattice;
			\item\label{sylow2b} $M$ is invertible as a $G_p$-lattice;
			\item\label{sylow3} $\on{Ext}^1_G(F,M)_{(p)}=0$ for every flasque $G$-lattice $F$.
		\end{enumerate}
	\end{lemma}

	\begin{proof}
		(\ref{sylow2b})$\Rightarrow$(\ref{sylow3}). Let $F$ be a flasque $G$-lattice, and consider the $G$-lattice $R:=\on{Hom}_{\Z}(F,M)$. By \cite[\S III, Proposition 2.2]{brown1994cohomology}, we have $\on{Ext}^1_G(F,M)=H^1(G,R)$. As $G_p$-lattices, $M$ is invertible and $F$ is flasque, so by \cite[Lemme 1]{colliot1977r} we have $H^1(G_p,R)=0$. The composition
		\[H^1(G,R)\xrightarrow{\on{Res}}H^1(G_p,R)\xrightarrow{\on{Cores}}H^1(G,R)\] is given by multiplication by $[G:G_p]$, so $\on{Ext}^1_G(F,M)$ is $[G:G_p]$-torsion. Since $G_p$ is a $p$-Sylow subgroup of $G$, $p$ does not divide $[G:G_p]$, hence $\on{Ext}^1_G(F,M)_{(p)}=0$.
		
		(\ref{sylow3})$\Rightarrow$(\ref{sylow1}). Consider a flasque resolution of the $G$-lattice $M$ \[0\to M\to P\to F\to 0.\] Recall that this means that the sequence is exact, $P$ is a permutation $G$-lattice, and $F$ is a flasque $G$-lattice. We have \[\on{Ext}^1_{\Z_{(p)}[G]}(F_{(p)},M_{(p)})\cong \on{Ext}^1_{\Z[G]}(F,M)_{(p)}=0,\] so the localized sequence \[0\to M_{(p)}\to P_{(p)}\to F_{(p)}\to 0\] splits, proving that $M_{(p)}$ is a direct summand of $P_{(p)}$. Therefore $M$ is $p$-invertible.
		
		The implication (\ref{sylow1})$\Rightarrow$(\ref{sylow2}) is obvious, and (\ref{sylow2})$\Rightarrow$(\ref{sylow2b}) follows from \Cref{firstprop}\ref{firstprop1b} and \ref{firstprop2}. 
	\end{proof}
	
	\section{The character lattice of a $p$-retract rational torus}\label{sec3}
	Let $k$ be a field, let $X$ be a normal equidimensional affine $k$-scheme with normal projective completion $\cl{X}$, and let $Z_1,\dots,Z_r$ be the irreducible components of the boundary $\cl{X}\setminus X$. Denote $\on{Div}_{\partial \cl{X}}\cl{X}:=\oplus_{i=1}^r \Z Z_i$. We have the following exact sequence:
	\begin{equation}\label{affinenormal} 0\to k'^*\to k[X]^*\to \on{Div}_{\partial \cl{X}}\cl{X}.\end{equation}
	Here $k':=k[\cl{X}]$ is the algebraic closure of $k$ in $k(X)$: it is finite-dimensional $k$-algebra. If $X$ is integral, it is a field. 
	
	If $G$ is a finite group and $M$ is a $G$-lattice, we denote by $[M]^{\on{fl}}$ the flasque class of $M$; see \cite[Lemme 5]{colliot1977r} (where it is denoted $\rho(M)$) or \cite[\S 2.7]{lorenz2006multiplicative}. If some $G$-lattice $F\in [M]^{\on{fl}}$ is invertible, or $p$-invertible, then the same is true for any $F'\in [M]^{\on{fl}}$; see \Cref{firstprop}\ref{firstprop4} for the case where $M$ is $p$-invertible. When this happens, we say that $[M]^{\on{fl}}$ is invertible, or $p$-invertible, respectively.
	
	By a theorem of Saltman \cite[Theorem 3.14(a)]{saltman1984retract}, a $k$-torus $T$ is retract rational if and only if $[\hat{T}]^{\on{fl}}$ is invertible; see also \cite[Lemma 9.5.4(b)]{lorenz2006multiplicative}. We now prove a $p$-local version of this result. In the course of the proof, we will make use of \cite[Theorem 2.3]{saltman1984retract}, hence we recall a piece of notation used there. If $G$ is a finite group and $f:M\to M'$ is a homomorphism of $G$-lattices, we write $\eta(f)=[0]$ if there exists a diagram with exact rows
	\begin{equation*}
	\begin{tikzcd}
	0 \arrow[r] & M \arrow[d, "f"] \arrow[r] & P \arrow[d] \arrow[r] & E \arrow[d,"g"] \arrow[r] & 0 \\
	0\arrow[r] & M' \arrow[r] & P' \arrow[r] & E' \arrow[r] &  0 
	\end{tikzcd}
	\end{equation*}\noindent 
	where $P$ and $P'$ are permutation $G$-lattices, $E$ and $E'$ are invertible $G$-lattices, and $g$ factors through a permutation $G$-lattice; see \cite[pp. 174-175]{saltman1984retract}. 
	
	\begin{prop}\label{character}
		Let $T$ be a torus over $k$. The following are equivalent:
		\begin{enumerate}
			\item\label{character1} $T$ is $p$-retract rational;
			\item\label{character2}  $[\hat{T}]^{\on{fl}}$ is $p$-invertible;
			\item\label{character3} there exists a commutative diagram of $\on{Gal}(k)$-lattices
			\begin{equation}\label{saltdiag}
			\begin{tikzcd}
			& \hat{T} \arrow[d, "\phi_{\hat{T},d}"] \arrow[r] & P \arrow[d] \\
			0\arrow[r] & \hat{T} \arrow[r] & M \arrow[r] & Q \arrow[r]&  0 
			\end{tikzcd}
			\end{equation}\noindent 
			where $d$ is not divisible by $p$, the bottom row is exact and $P$ and $Q$ are permutation lattices.
		\end{enumerate}
	\end{prop}	
	
	\begin{proof}
		Let $L$ be a splitting field of $T$, and set $G:=\on{Gal}(L/k)$.
		
		(\ref{character2})$\Rightarrow$(\ref{character3}). Assume that the $G$-lattice $[\hat{T}]^{\on{fl}}$ is $p$-invertible, and choose a flasque resolution $0\to \hat{T}\to P\to E\to 0$ of $\hat{T}$, where $E$ is $p$-invertible. There exists $d\in \Z$ not divisible by $p$ such that $\phi_{E,d}$ factors through a permutation $G$-lattice $Q$. The diagram
		\begin{equation*}
		\begin{tikzcd}
		0 \arrow[r] & \hat{T} \arrow[d, "\phi_{\hat{T},d}"] \arrow[r] & P \arrow[d,"\phi_{P,d}"] \arrow[r] & E \arrow[d,"\phi_{E,d}"] \arrow[r] & 0 \\
		0\arrow[r] & \hat{T} \arrow[r] & P \arrow[r] & E \arrow[r] &  0 
		\end{tikzcd}
		\end{equation*}\noindent
		then shows that $\eta(\phi_{\hat{T},d})=[0]$. By \cite[Theorem 2.3]{saltman1984retract}, a diagram of type (\ref{saltdiag}) exists.

		(\ref{character3})$\Rightarrow$(\ref{character2}). If we have a diagram of type (\ref{saltdiag}), then by \cite[Theorem 2.3]{saltman1984retract} we have $\eta(\phi_{\hat{T},d})=[0]$. This means that there exists a diagram
		\begin{equation}\label{diagramma1}
		\begin{tikzcd}
		0\arrow[r] & \hat{T} \arrow[r,"\iota_1"] \arrow[d, "\phi_{\hat{T},d}"] & P_1 \arrow[r,"\pi_1"] \arrow[d,"\rho_1"]  & E_1 \arrow[r] \arrow[d,"\psi_1"] &  0 \\
		0\arrow[r] & \hat{T} \arrow[r,"\iota_2"] & P_2 \arrow[r,"\pi_2"] & E_2 \arrow[r]&  0 
		\end{tikzcd}
		\end{equation} 
		where the rows are flasque resolutions, and $\psi_1$ factors through a permutation $G$-lattice $Q'$. By \cite[Lemme 4]{colliot1977r}, $\iota_1$ factors through $P_2$, so we have a diagram
		\begin{equation}\label{diagramma2}
		\begin{tikzcd}
		0\arrow[r] & \hat{T} \arrow[r,"\iota_2"] \arrow[d,equal] & P_2 \arrow[r,"\pi_2"] \arrow[d,"\rho_2"]  & E_2 \arrow[r] \arrow[d,"\psi_2"] &  0 \\
		0\arrow[r] & \hat{T} \arrow[r,"\iota_1"] & P_1 \arrow[r,"\pi_1"] & E_1 \arrow[r]&  0 
		\end{tikzcd}
		\end{equation} 
		Let $\rho:=\rho_2\circ\rho_1$ and $\psi:=\psi_2\circ \psi_1$. The bottom row of (\ref{diagramma1}) and the top row of (\ref{diagramma2}) coincide. Joining the two previous diagrams along said rows, we obtain 
		\begin{equation*}
		\begin{tikzcd}
		0\arrow[r] & \hat{T} \arrow[r,"\iota_1"] \arrow[d, "\phi_{\hat{T},d}"] & P_1 \arrow[r,"\pi_1"] \arrow[d,"\rho"]  & E_1 \arrow[r] \arrow[d,"\psi"] &  0 \\
		0\arrow[r] & \hat{T} \arrow[r,"\iota_1"] & P_1 \arrow[r,"\pi_1"] & E_1 \arrow[r]&  0. 
		\end{tikzcd}
		\end{equation*} 	
		We also have the following diagram:
		\begin{equation*}
		\begin{tikzcd}
		0\arrow[r] & \hat{T} \arrow[r,"\iota_1"] \arrow[d, "\phi_{\hat{T},d}"] & P_1 \arrow[r,"\pi_1"] \arrow[d,"\phi_{P_1,d}"]  & E_1 \arrow[r] \arrow[d,"\phi_{E_1,d}"] &  0 \\
		0\arrow[r] & \hat{T} \arrow[r,"\iota_1"] & P_1 \arrow[r,"\pi_1"] & E_1 \arrow[r]&  0. 
		\end{tikzcd}
		\end{equation*} 
		Combining these two diagrams, it is immediate to deduce that $\phi_{P_1,d}-\rho$ is zero on $\on{im}\iota_1$, and that the square
		\begin{equation*}
		\begin{tikzcd}
		P_1 \arrow[r,"\pi_1"] \arrow[d,swap,"\phi_{P_1,d}-\rho"]  & E_1 \arrow[d,"\phi_{E_1,d}-\psi"]\\
		P_1 \arrow[r,"\pi_1"] & E_1 
		\end{tikzcd}
		\end{equation*} 
		commutes. Since $\phi_{P_1,d}-\rho$ is zero on $\on{im}\iota_1=\on{ker}\pi_1$, there exists $f:E_1\to P_1$ making the diagram
		\begin{equation*}
		\begin{tikzcd}
		P_1 \arrow[r,"\pi_1"] \arrow[d,swap,"\phi_{P_1,d}-\rho"]  & E_1 \arrow[d,"\phi_{E_1,d}-\psi"] \arrow[dl,swap,"f"]\\
		P_1 \arrow[r,"\pi_1"] & E_1 
		\end{tikzcd}
		\end{equation*} 
		commute.
		
		Recall that $\psi_1$ factors through a permutation $G$-lattice $Q'$. Since $\psi$ factors through $\psi_1$, it also factors through $Q'$.
		We thus have commutative triangles
		\begin{equation*}
		\begin{tikzcd}
		E_1 \arrow[r] \arrow[dr,swap,"\psi"]& Q' \arrow[d]\\
		& E_1 \end{tikzcd}\qquad
		\begin{tikzcd}
		E_1 \arrow[r,"f"] \arrow[dr,swap,"\phi_{E_1,d}-\psi"]& P_1 \arrow[d,"\pi_1"]\\
		& E_1. 
		\end{tikzcd}
		\end{equation*}\noindent 
		Therefore $\phi_{E_1,d}=(\phi_{E_1,d}-\psi)+\psi$ factors through $P_1\oplus Q'$, which implies that $E_1\in[\hat{T}]^{\on{fl}}$ is $p$-invertible.

		(\ref{character3})$\Rightarrow$(\ref{character1}). By \cite[Theorem 2.3]{saltman1984retract}, the existence of a diagram of $G$-lattices of type (\ref{saltdiag}) implies that $\eta(\phi_{\hat{T},d})=[0]$. By \cite[Theorem 3.14(b)]{saltman1984retract}, this implies that the ring homomorphism $f_d:L[\hat{T}]^G\to L[\hat{T}]^G$ induced by $\phi_{\hat{T},d}$ factors rationally, that is, there exist an integer $n\geq 1$, a non-zero polynomial $w\in k[x_1,\dots,x_n]$, a non-zero $u\in L[\hat{T}]^G$, and maps $\alpha:L[\hat{T}]^G\to k[x_1,\dots,x_n][1/w]$ and $\beta:k[x_1,\dots,x_n][1/w]\to L[\hat{T}]^G[1/u]$ such that $f_d=\beta\circ\alpha$; see \cite[Definition 3.2]{saltman1984retract}. Geometrically, this means that the map $T\to T$ given by $t\mapsto t^d$ factors rationally through $\A^n$. Since the degree of this map is coprime to $p$, we conclude that $T$ is $p$-retract rational.
		
		(\ref{character1})$\Rightarrow$(\ref{character3}). Assume that we are given a diagram
		\[
		\begin{tikzcd}
		& X \arrow[d, "\phi"] \arrow[dl,"s"] \\
		\A^n \arrow[r, dashrightarrow, "\pi"] & T  
		\end{tikzcd}
		\]\noindent
		where $X$ is a variety and $\phi$ is a finite dominant map of degree $d$ not divisible by $p$. We have an induced map $\pi^*:k[T]\to k[x_1,\dots,x_n][1/w]$, where $w\neq 0$ is a polynomial in $x_1,\dots,x_n$. 
		
		Let $\psi:={\phi}_L:{X}_L\to {T}_L$. We have a $G$-equivariant homomorphism $\psi^*:L[T]^*\to L[X]^*$ and a $G$-equivariant norm map $\psi_*:L[X]^*\to L[T]^*$, induced by the norm map of field theory $L(X)^*\to L(T)^*$. The composition $\psi_*\circ \psi^*$ is given by $s\mapsto s^d$. 
		
		We have $L[T]^*/L^*=\hat{T}$. Let $k'$ be the algebraic closure of $k$ in $k(X)$: it is a finite extension of $k$. By (\ref{affinenormal}), $N':=L[X]^*/(k'\otimes L)^*$ is a $G$-sublattice of $\on{Div}_{\partial \cl{X}_L}\cl{X}_L$, where $\cl{X}$ is a normal projective compactification of $X$. We obtain the following commutative diagrams of $G$-equivariant maps with exact rows:
		\[
		\begin{tikzcd}
		0 \arrow[r] & L^*\arrow[r] \arrow[d] & L[T]^*\arrow[r] \arrow[d, "\psi^*"]  & \hat{T} \arrow[d, "\psi^*"]\arrow[r] & 0 \\
		0 \arrow[r] & (k'\otimes L)^*\arrow[r] & L[X]^*\arrow[r] & N' \arrow[r] & 0 
		\end{tikzcd}
		\]\noindent
		and
		\[
		\begin{tikzcd}
		0 \arrow[r] & L^*\arrow[r] & L[T]^*\arrow[r]   & \hat{T} \arrow[r] & 0 \\
		0 \arrow[r] & (k'\otimes L)^*\arrow[r] \arrow[u] & L[X]^*\arrow[r] \arrow[u, "\psi_*"] & N' \arrow[u, "\psi_*"] \arrow[r] & 0.
		\end{tikzcd}
		\]\noindent
		The composition $\psi_*\circ \psi^*:\hat{T}\to \hat{T}$ equals $\phi_{\hat{T},d}$.
		Let $P:=L[x_1,\dots,x_n][1/w]^*/L^*$. Write $w=c\cdot w_1^{a_1}\cdot\ldots\cdot w_r^{a_r}$, where $c\in L^*$, the $w_i$ are irreducible polynomials in $L[x_1,\dots,x_n]$, and $w_i$ is not a scalar multiple of $w_j$ when $i\neq j$. The cosets of the $w_i$ modulo $L^*$ freely generate $P$ as a $\Z$-module, and they are permuted by the $G$-action, hence $P$ is a permutation $G$-lattice. 
		
		The ring map $\pi^*_L:L[T]\to L[x_1,\dots,x_n][1/w]$ induces a group homomorphism $L[T]^*\to L[x_1,\dots,x_n][1/w]^*$, hence an injective $G$-homomorphism $\pi':\hat{T}\to P$. The restriction of $s:X\to \A^n$ to the domain of definition of $\pi$ gives a map $k[x_1,\dots,x_n][1/w]\to k[X][1/u]$, for some non-zero $u\in k[X]$. This in turn induces a homomorphism of $G$-lattices $P\to N:=L[X][1/u]^*/(k'\otimes L)^*$. We have an exact sequence \[0\to N'\to N\to R\to 0\] for some $G$-lattice $R$. We have constructed the following commutative diagram:
		\begin{equation}\label{diag1}
		\begin{tikzcd}
		& \hat{T} \arrow[d, "\psi^*"] \arrow[r] & P  \arrow[d] \\
		0\arrow[r] & N' \arrow[r] & N\arrow[r] & R \arrow[r]&  0. 
		\end{tikzcd}
		\end{equation}
		The norm map $L(X)^*\to L(T)^*$ induces a commutative square
		\begin{equation*}
		\begin{tikzcd}
		L[X]^* \arrow[d, "\psi_*"] \arrow[r] & L[X][1/u]^* \arrow[d,"\psi_*"] \\
		L[T]^* \arrow[r] & L[T][1/\psi_*(u)]^*.
		\end{tikzcd}
		\end{equation*}\noindent 
		Since $u\in k[X]$, we have $\psi_*(u)\in k[T]$. We obtain a commutative square
		\begin{equation*}
		\begin{tikzcd}
		L[X]^*/(k'\otimes L)^* \arrow[d, "\psi_*"] \arrow[r] & L[X][1/u]^*/(k'\otimes L)^* \arrow[d,"\psi_*"] \\
		L[T]^*/L^* \arrow[r] & L[T][1/\psi_*(u)]^*/L^*.
		\end{tikzcd}
		\end{equation*}\noindent 
		Recall that $\hat{T}= L[T]^*/L^*$, $N'=L[X]^*/(k'\otimes L)^*$ and $N=L[X][1/u]^*/(k'\otimes L)^*$, hence the above diagram is
		\begin{equation*}
		\begin{tikzcd}
		N' \arrow[d, "\psi_*"] \arrow[r] & N \arrow[d,"\psi_*"] \\
		\hat{T} \arrow[r] & M,
		\end{tikzcd}
		\end{equation*}\noindent 
		where we define $M:=L[T][1/\psi_*(u)]^*/L^*$. The homomorphism $\hat{T}\to M$ is injective, and its cokernel is $Q=\on{Div}_{Z_L}T_L$, where $Z$ is the zero locus of $\psi_*(u)$ inside $T$. In other words, $Q$ is the permutation $G$-lattice generated by the irreducible factors of $\psi_*(u)$ in $L[T]$. 
		By construction, we have a commutative diagram with exact rows
		\begin{equation}\label{diag2}
		\begin{tikzcd}
		0\arrow[r] & {N'} \arrow[r] \arrow[d] & N \arrow[r] \arrow[d]  & R \arrow[r] \arrow[d] &  0 \\
		0\arrow[r] & {\hat{T}} \arrow[r] & M \arrow[r] & Q \arrow[r]&  0. 
		\end{tikzcd}
		\end{equation} 
		Combining (\ref{diag1}) and (\ref{diag2}) gives a diagram of type (\ref{saltdiag}). By construction, $\psi_*\circ\psi^*=\phi_{\hat{T},d}$, where $d=\on{deg}\phi$ is coprime with $p$, and $P$ and $Q$ are permutation $G$-lattices, hence (\ref{character3}) follows.
	\end{proof}
	
	\begin{cor}\label{sylowfortori}
		Let $T$ be a $k$-torus with splitting field $L$ and $G:=\on{Gal}(L/k)$. For a prime $p$, denote by $G_p$ a $p$-Sylow subgroup of $G$, and let $k_p:=L^{G_p}$ be the fixed field of $G_p$. Then $T$ is $p$-retract rational over $k$ if and only $T_{k_p}$ is retract rational over $k_p$.
	\end{cor}
	
	\begin{proof}
		The character lattice of $T_{k_p}$ is $\hat{T}$ viewed as a $G_p$-lattice via restriction. Fix a $G$-lattice $F\in [\hat{T}]^{\on{fl}}$. A flasque resolution of $M$ as a $G$-lattice is in particular a flasque resolution of $M$ as a $G_p$-lattice, hence $F\in [\hat{T}_{k_p}]^{\on{fl}}$ when $F$ is viewed as a $G_p$-lattice.
		
		By \Cref{character}, $T$ is $p$-retract rational if and only if $F$ is $p$-invertible. By \Cref{sylow}, this is the same as $F$ being invertible as a $G_p$-lattice. By \cite[Lemma 9.5.4(b)]{lorenz2006multiplicative}, this is equivalent to the retract rationality of $T_{k_p}$.
	\end{proof}

	\section{Proof of Theorem \ref{local} and Corollary \ref{bglocal}}\label{sec4}
	
	\begin{proof}[Proof of \Cref{local}]
		If $T$ is retract rational, it is also $p$-retract rational for every prime $p$. Conversely, if $T$ is $p$-retract rational for every $p$, then by \Cref{character}, $[\hat{T}]^{\on{fl}}$ is $p$-invertible for every $p$. By \Cref{firstprop}\ref{firstprop2}, $[\hat{T}]^{\on{fl}}$ is invertible. By \cite[Lemma 9.5.4(b)]{lorenz2006multiplicative}, $T$ is retract rational.	
	\end{proof}
	
	Let $G$ be a group of multiplicative type. Embedding $G$ in some quasi-split torus $S$, one obtains a short exact sequence \begin{equation}\label{bgres}1\to G\to S\to T\to 1\end{equation} where $T$ is a torus and $S$ is a quasi-split torus. Since $S$ is quasi-split, it is an open subset of some generically free linear $T$-representation. By definition, $BG$ is retract rational ($p$-retract rational) if and only if $T$ is; see the introduction. Recall that the rationality ($p$-retract rationality) of $T$ does not depend on the choice of the embedding $G\hookrightarrow S$, by the no-name lemma \cite[Lemma 2.1]{reichstein2006birational}.
	
	\begin{proof}[Proof of \Cref{bglocal}]
		Let $G$ be a group of multiplicative type, and consider a sequence of type (\ref{bgres}) for $G$. By assumption, $T$ is $p$-retract rational for every prime $p$, thus it is retract rational by \Cref{local}. This means precisely that $BG$ is retract rational, as desired. 
	\end{proof}

	\section{Proof of \Cref{Theorem 1.3}}\label{sec5}
	
	Let $G$ be a finite group. Recall that the augmentation map $\epsilon:\Z[G]\to \Z$ is defined by $\epsilon(g)=1$ for every $g\in G$. Set $I_G:=\on{ker}\epsilon$ and let $J_G:=I_G'$ be the dual of $I_G$. Equivalently, $J_G$ is the cokernel of the norm map $N:\Z\to \Z[G]$ given by $N(1)=\sum_{g\in G}g$. If $L/k$ is a Galois extension of Galois group $G$, then $J_G$ is isomorphic to the character lattice of the norm one torus $R^{(1)}_{L/k}(\G_m)$.
	
	The following is a $p$-local version of \cite[Proposition 2]{colliot1977r}.
	\begin{prop}\label{norminv}
		Let $G$ be a finite group, and let $G_p$ be a $p$-Sylow subgroup of $G$. The following conditions are equivalent:
		\begin{enumerate}
			\item\label{norminv1} $G_p$ is cyclic;
			\item\label{norminv2} every flasque $G$-lattice is $p$-invertible;
			\item\label{norminv3} $[J_G]^{\on{fl}}$ is $p$-invertible.
		\end{enumerate}
	\end{prop}
	
	\begin{proof}
		By definition $[J_G]^{\on{fl}}$ is flasque, so (\ref{norminv2}) implies (\ref{norminv3}). 
		
		(\ref{norminv3})$\Rightarrow$(\ref{norminv1}). By \Cref{sylow}, $[J_G]^{\on{fl}}$ is invertible as a $G_p$-lattice. By \cite[Proposition 2]{colliot1977r}, $G_p$ is cyclic.
		
		(\ref{norminv1})$\Rightarrow$(\ref{norminv2}). By \Cref{sylow}, we may assume that $G$ is a cyclic $p$-group. In this case, by \cite[Proposition 2]{colliot1977r} every flasque $G$-lattice is invertible, hence $p$-invertible. 
	\end{proof} 
	
	\begin{proof}[Proof of \Cref{Theorem 1.3}]
		Let $L/k$ be a finite Galois extension. By \Cref{character} and \Cref{norminv} the torus $R^{(1)}_{L/k}(\G_m)$ is $p$-retract rational if and only if the $p$-Sylow subgroups of $\on{Gal}(L/k)$ are cyclic. 
		
		Let $S$ be a finite set of primes. Set $G:=\prod_{p\in S} (\Z/p\Z)^2$, and consider a Galois field extension $L/k$ such that $\on{Gal}(L/k)\cong G$. It is well known that such extensions exist, even for $k=\Q$; see e.g. \cite[Chapter VI, Exercise 23]{lang2002algebra}. We conclude that $T:=R^{(1)}_{L/k}(\G_m)$ is $p$-retract rational if and only if $p \not \in S$, as claimed.
	\end{proof}
	
	\begin{lemma}\label{stackprr}
		Let $T$ be a $k$-torus, and let $T'$ be its dual. For every prime $p$, $BT$ is $p$-retract rational if and only if $T'$ is $p$-retract rational. 
	\end{lemma}
	
	\begin{proof}
		Let $L$ be the splitting field of $T$, $G=\on{Gal}(L/k)$, $G_p$ a $p$-Sylow subgroup of $G$, and $k_p=L^{G_p}$ the fixed field of $G_p$. By \Cref{sylowfortori}, $T'$ is $p$-retract rational if and only if $T'_{k_p}$ is retract rational over $k_p$.
		Consider an exact sequence \[1\to T\to S\to Q\to 1\] where $S$ is quasi-split. Since $BT$ is $p$-retract rational if and only if $Q$ is, from \Cref{sylowfortori} we deduce that $BT$ is $p$-retract rational if and only if $BT_{k_p}$ is retract rational over $k_p$. By \Cref{stack}, this is equivalent to $T'_{k_p}$ being retract rational over $k_p$.
	\end{proof}

	\begin{cor}\label{stackspecified}
		For every finite set of primes $S$, there exists an algebraic torus $T$ such that the classifying stack $BT$ is $p$-retract rational if and only if $p\notin S$.
	\end{cor}	
	
	\begin{proof}
		Combine \Cref{Theorem 1.3} and \Cref{stackprr}.
	\end{proof}	
	
	\section{Appendix}
	The following result was stated by Merkurjev. The proof given below develops a message of Colliot-Th\'el\`ene to Merkurjev (11th October 2015).
	\begin{prop}\label{stack}
		Let $T$ be an algebraic $k$-torus, and let $T'$ be its dual. Then:
		\begin{enumerate}[label=(\alph*)]
			\item\label{stackstab} $BT$ is stably rational if and only if $T'$ is,
			\item\label{stackret} $BT$ is retract rational if and only if $T'$ is. 
		\end{enumerate} 
	\end{prop}
	
	\begin{proof}
		If $T'$ is stably rational, by \cite[Theorem 2 p. 52]{voskresenskii2011algebraic} there exists a short exact sequence \[1\to S_1\to S_2\to T'\to 1\] where $S_1$ and $S_2$ are quasi-split. Recall that a quasi-split torus is isomorphic to its dual, and is an open subset of some affine space, hence it is a rational variety. The dual sequence gives a versal $T$-torsor $S'_2\to S'_1$ with rational base, hence $BT$ is stably rational.
		
		If $T'$ is retract rational, by \cite[Lemma 9.5.4(b)]{lorenz2006multiplicative} the flasque class of $\hat{T}'$ is invertible, hence we have a sequence \[1\to U\to S\to T'\to 1\] where $S$ is quasi-split and $U$ is an invertible torus, i.e. there exists an isomorphism of tori $U\times Q\cong P$, where $Q$ is a torus and $P$ is a quasi-split torus. Dualization gives $U'\times Q'\cong P'$, hence $U'$ is retract rational. The dual $1\to T\to S'\to U'\to 1$ of the previous sequence yields a versal $T$-torsor $S'\to U'$ over a retract rational base, hence $BT$ is retract rational.  
		
		Before proving the converse, we need some preparation. Let $L$ be the splitting field of $T$, and let $G:=\on{Gal}(L/k)$. Choose a coflasque resolution of $\hat{T}$ \begin{equation}\label{stack1}0\to \hat{Q}\to \hat{R}\to \hat{T}\to 0,\end{equation} that is, $\hat{R}$ is a permutation $G$-lattice and $\hat{Q}$ is coflasque, and a flasque resolution of $\hat{Q}$ \begin{equation}\label{stack2}0\to \hat{Q}\xrightarrow{\alpha} \hat{S}\xrightarrow{\alpha'} \hat{F}\to 0.\end{equation} Since $R\to Q$ is a versal $T$-torsor, $BT$ is stably rational (resp. retract rational) if and only if $Q$ is. Applying \cite[Theorem 2 p. 52]{voskresenskii2011algebraic} (resp. \cite[Lemma 9.5.4(b)]{lorenz2006multiplicative}) to sequence (\ref{stack2}), this is in turn equivalent to $\hat{F}$ being stably permutation (resp. invertible). 
		
		\ref{stackstab}. Assume that $BT$ is stably rational. By (\ref{stack1}), this implies that $Q$ is stably rational. Since $\hat{F}\in [\hat{Q}]^{\on{fl}}$, by \cite[Theorem 2 p. 52]{voskresenskii2011algebraic} there is a permutation $G$-lattice $\hat{S}_1$ such that $\hat{F}\oplus \hat{S}_1$ is a permutation $G$-lattice. Then the sequence \[0\to \hat{Q}\xrightarrow{(\alpha,0)} \hat{S}\oplus\hat{S}_1\xrightarrow{(\alpha',\on{id})} \hat{F}\oplus \hat{S}_1\to 0\] is another flasque resolution of $\hat{Q}$. We may thus assume that $\hat{F}$ is a permutation $G$-lattice. By \cite[Lemme 1(viii)]{colliot1977r}, it follows that $\hat{S}\cong \hat{Q}\oplus \hat{F}$, and we obtain $\hat{S}\cong \hat{Q}'\oplus \hat{F}$ by dualization. Consider the sequence \[0\to \hat{T}'\xrightarrow{\beta}\hat{R}'\xrightarrow{\beta'}\hat{Q}'\to 0\] dual to (\ref{stack1}). Then the sequence \[0\to \hat{T}'\xrightarrow{(\beta,0)}\hat{R}'\oplus\hat{F}\xrightarrow{(\beta',\on{id})}\hat{Q}'\oplus\hat{F}\to 0\] is also exact. By \cite[Theorem 2 p. 52]{voskresenskii2011algebraic}, we conclude that $T'$ is stably rational.
		
		\ref{stackret}. We have seen that $BT$ is retract rational if and only $\hat{F}$ is invertible. Assume that $BT$ is retract rational. Applying \cite[Lemme 1(vii)']{colliot1977r} to (\ref{stack2}) we have $\hat{S}\cong \hat{Q}\oplus \hat{F}$, hence $\hat{S}'\cong\hat{Q}'\oplus\hat{F}'$, showing that $\hat{Q}'$ is invertible. Since $\hat{Q}'\in [\hat{T}']^{\on{fl}}$, by \cite[Lemma 9.5.4(b)]{lorenz2006multiplicative} we conclude that $T'$ is retract rational. 
	\end{proof}
	
	\section*{Acknowledgments}
	I would like to thank my advisor Zinovy Reichstein for helpful comments, and Jean-Louis Colliot-Th\'el\`ene for communicating \Cref{stack} to me. I thank the referee for greatly improving the exposition of this paper.

\end{document}